\documentclass[a4paper, 11pt]{article}
\usepackage[margin=3cm]{geometry}
\usepackage{mathtools} 
\usepackage{booktabs} 
\usepackage{tikz} 
\usepackage[ruled,vlined]{algorithm2e}
\usepackage{siunitx}
\usepackage{amsmath, amsthm, amssymb, bbm}
\usepackage{natbib}

\DeclareMathOperator{\pa}{pa}

\DeclareMathOperator{\an}{an}
\DeclareMathOperator{\de}{de}

\DeclareMathOperator{\cn}{cn}

\DeclareMathOperator{\Do}{do}

\DeclareMathOperator{\forb}{forb}
\DeclareMathOperator{\rch}{rch}

\DeclareMathOperator{\clr}{clr}
\DeclareMathOperator{\opadj}{O}

\DeclareMathOperator{\uniform}{uniform}

\DeclarePairedDelimiterX{\infdivx}[2]{(}{)}{%
  #1\;\delimsize\|\;#2%
}

\newcommand\indep{\protect\mathpalette{\protect\independenT}{\perp}}
\def\independenT#1#2{\mathrel{\rlap{$#1#2$}\mkern2mu{#1#2}}}

\theoremstyle{plain}
\newtheorem{theorem}{Theorem}[section]
\newtheorem{corollary}{Corollary}[theorem]
\newtheorem{lemma}[theorem]{Lemma}
\newtheorem{proposition}[theorem]{Proposition}

\newtheorem{definition}[theorem]{Definition}
\theoremstyle{remark}

\newtheorem{remark}{Remark}
\theoremstyle{definition}
\newtheorem{example}{Example}[section]

\newcommand{\G}{{\cal G}}

\tikzset{deg/.style={->, very thick, color=blue}}
\tikzset{degl/.style={->, very thick, color=red}}
\tikzset{beg/.style={<->, very thick, color=red}}

\usetikzlibrary{shapes, arrows}
\usepackage{graphicx}
\usepackage{multicol}
\usepackage{multirow}
\usepackage{tabularx}
\usepackage{enumerate}

\usetikzlibrary{chains,
                positioning}
\usepackage{enumitem}
\usepackage{float}
\usepackage{bbm}
\usepackage{nccmath}
\usepackage{babel}

\usepackage{lipsum}
 \usepackage[justification=centering]{caption}
\usepackage{rotating}

\title{Selecting valid adjustment sets with uncertain causal graphs}

\author{{\bf Zhongyi Hu}\\Department of Mathematics\\Vrije Universiteit\\z.hu@vu.nl \and 
{\bf St\'{e}phanie~L.~van~der~Pas}\\Department of Mathematics\\Vrije Universiteit\\s.l.vander.pas@vu.nl}

\begin{document}

\allowdisplaybreaks

\maketitle
  \begin{abstract}
{Precise knowledge of causal directed acyclic graphs (DAGs) is assumed for standard approaches towards valid adjustment set selection for unbiased estimation, but in practice, the DAG is often inferred from data or expert knowledge, introducing uncertainty. We present techniques to identify valid adjustment sets despite potential errors in the estimated causal graph. Specifically, we assume that only the skeleton of the DAG is known. Under a Bayesian framework, we place a prior on graphs and wish to sample graphs and compute the posterior probability of each set being valid; however, directly doing so is inefficient as the number of sets grows exponentially with the number of nodes in the DAG. We develop theory and techniques so that a limited number of sets are tested while the probability of finding valid adjustment sets remains high. Empirical results demonstrate the effectiveness of the method.}
\end{abstract}

\section{Introduction}
When the causal \emph{directed acyclic graph} (DAG) \citep{pearl2000models} is unknown, standard adjustment set selection procedures designed for a known DAG but applied on an inferred DAG may result in invalid adjustment sets without the user realizing. Subsequent causal effect estimates may then be biased. Our goal is to select a set of variables $X_A$ that is a \emph{valid adjustment set}, that is, conditioning on them along with treatment $T$, the expected mean of the outcome $Y$, $\mathbb{E}[Y \mid T,X_A]$ can be used for unbiased estimation of the expected total effect of $T$ on $Y$, $\mathbb{E}[Y \mid \Do(T)]$ \citep{pearl2000models}. We develop methods for the practical situation where the DAG is inferred from data or expert knowledge. Our methods will find with high probability valid adjustment sets that are robust to some common types of graph misspecification.


\begin{figure}
\centering
  \begin{tikzpicture}
  [rv/.style={circle, draw, thick, minimum size=6mm, inner sep=0.8mm}, node distance=14mm, >=stealth]
  \pgfsetarrows{latex-latex};
\begin{scope}
  \node[rv]  (1)            {$T$};
  \node[rv, right of=1] (2) {$1$};
  \node[rv, right of=2] (3) {$2$};
  \node[rv, right of=3] (4) {$Y$};
  \node[rv, above of=2] (5) {$3$};
  \node[rv, below of=3] (6) {$4$};
  \node[rv, above of=1] (7) {$5$};
  \node[rv, above of=5] (8) {$6$};
  \node[rv, right of=5] (9) {$7$};
  \node[rv, above of=7] (10) {$8$};
  
  \draw[->, very thick, blue] (1) -- (2);
  \draw[->, very thick, blue] (2) -- (3);
  \draw[->, very thick, blue] (3) -- (4);
  \draw[->, very thick, blue] (3) -- (6);
  \draw[->,very thick, blue] (5) -- (2);

  \draw[->, very thick, blue] (7) -- (1);
  \draw[->, very thick, blue] (10) -- (7);
  \draw[->, very thick, blue] (7) -- (8);
  \draw[->, very thick, blue] (9) -- (8);
  \draw[->, very thick, blue] (9) -- (4);
  \end{scope}
\end{tikzpicture}
\caption{Graph for adjustment set examples. }
\label{fig: adjustment set examples}
\end{figure}
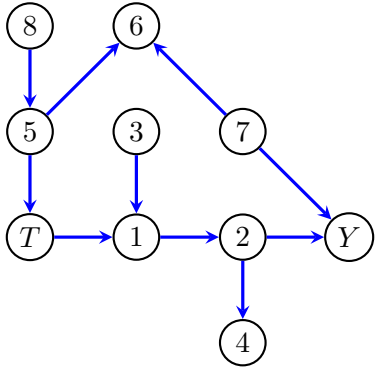

\textbf{Related work on adjustment sets when the causal DAG is known}.  Consider the causal DAG in Figure \ref{fig: adjustment set examples} with treatment $T$ and outcome $Y$. From the famous \emph{back-door criterion} \citep{pearl2000models}, the valid adjustment sets are any set from $\emptyset$, $\{5\}$, $\{7\}$, $\{5,7\}$, $\{5,6\}$, $\{6,7\}$, $\{5,6,7\}$, and combine with any subset of $\{3,8\}$. \citet{van2019separators} provide algorithms for computing all valid adjustment sets given a causal DAG and also show that all adjustment sets can be found with polynomial delay.

Different adjustment sets have different efficiency in terms of the variance of the estimated causal effect. \citet{henckel2022graphical} and \citet{smucler2022efficient} prove and construct the most efficient valid (`optimal') adjustment sets in the linear Gaussian model and the non-parametric causal model, respectively. For the causal DAG in Figure \ref{fig: adjustment set examples}, the optimal adjustment set is $\{3,7\}$. 

\textbf{Related work on adjustment sets when the causal DAG is unknown}. Often, only a Markov equivalence class of a DAG is known, which is usually represented by \emph{completely partially oriented directed acyclic graphs} (CPDAGs). Instead of directed edges, those graphs could have undirected edges, where an undirected edge represents that there are two DAGs lying in the same Markov equivalence class with this edge oriented differently. The theory of complete characterization of valid adjustment sets and optimal adjustment has also been extended to CPDAGs \citep{perkovi2018complete,henckel2022graphical}.

When only observational data are available, the data can be put in some graph learning algorithm, for example, PC \citep{pearl2000models}, GES \citep{chickering2002optimal}, LiNGAM \citep{shimizu2006linear}. Depending on the parametric assumption and the class of graphical model used, the output is either a DAG or a CPDAG. Then the methods mentioned above for known DAGs can be applied to compute valid adjustment sets. However, the actual validity of these sets depends on the accuracy of the output of the graph learning method, which is sensitive to the size of the data and the hyperparameters of the algorithms. Although there have been some results on the consistency of these algorithms \citep{kalisch2007estimating}, in practice, the conditions required for consistency are often difficult to satisfy or verify, and the output often does not match the true underlying graph. As a result, the corresponding valid adjustment sets in the estimated graph may not be valid in the underlying true graph. 

Little work has been done on computing adjustment sets together with consideration of the uncertainty of causal graphs, although some work has been done for situations where partial information on the DAG or data from multiple environments is available \citep{shah2022finding,shah2023front,shi2021invariant,de2025doubly}.

\textbf{Our contribution}. In this work, we present results on selecting valid adjustment sets in an estimated graph, while taking into account that the estimated graph could be incorrect. Theory about how the set of valid adjustment changes as the graph changes is developed, as well as efficient methods for computing valid adjustment sets in the face of uncertainty.  

Our work differs from \citet{vander2011new}, which has a different setting but a similar aim, namely to select a valid adjustment when complete knowledge of the causal structure is absent. They propose a selection criterion that includes all the variables that are either cause of treatment or outcome or both and show that if any subset of the observed variables is a valid adjustment set, then this set is a valid adjustment set. We need skeleton of the causal graph and do not assume any causal relation among the variables.

\textbf{Assumptions}. Some of our results are subject to certain graphical restrictions. In this work, we focus on one treatment and one outcome, with a known skeleton. There is no prior work on this topic on estimating valid adjustment sets by sampling DAGs, and the difficulty lies in the complexity and the large number of potential valid adjustment sets. 

The assumption of a known skeleton is reasonable because practically, most mistakes in graph learning algorithms come from getting the direction of edges wrong, not from missing the edges themselves \citep{tsamardinos2006max,mwebaze2010fast}; the direction of the edges often relies on detecting associations between variables. In addition, in some constraints-based methods, the skeleton is learned first and edge directions are added later, making the assignment of directions more error prone. In real-world situations, experts may know which variables are directly connected, but may not be sure about the direction of the relationship. We implement the choice of using an estimated skeleton from GES \citep{chickering2002optimal} in our algorithms. 

The intuition behind our approach is that similar graphs can possibly admit the same valid adjustment sets, and we can use some sampling technique to gather graphs that share some similarity to the true graph. We will now use some experimental and numerical results to demonstrate this motivation more clearly. The \emph{Structural Hamming Distance} (SHD) is a measure of the similarity of the graphs, which is defined as the number of differences in edge marks in two graphs. For each graph size $n \in \{6,7,8,9\}$ 100 repeats of the following procedure were made. First, a baseline DAG of size $n$ was generated with two nodes dedicated to treatment $T$ and result $Y$. Then another 300 DAGs of the same size were randomly generated and the SHD to the baseline DAG was computed. For each of the 300 DAGs, all valid adjustment sets were calculated and the proportion of such sets that was also valid in the baseline DAG was recorded as `precision’. The results (Figure \ref{fig: precision plot}) show that the more similar the DAGs are, as indicated by a lower SHD, the more similar the sets of valid adjustment sets are, as indicated by a higher precision.

\begin{figure}
    \centering
    \includegraphics[width=0.6\linewidth]{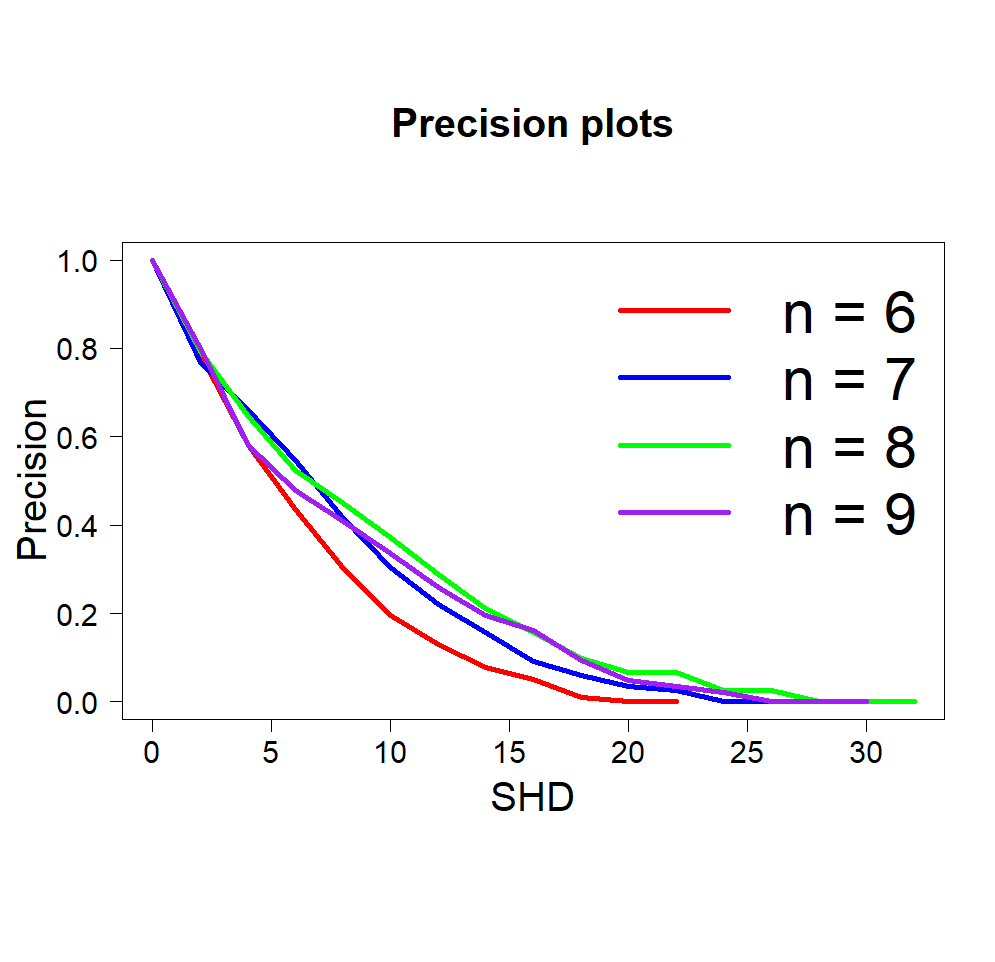}
    \caption{Precision plot showing that similar graphs (in terms of SHD) tend to have similar valid adjustment sets.}
    \label{fig: precision plot}
\end{figure}

Motivated by this intuition, we deduce theoretical results  about the stability of valid adjustment sets when the graph changes  in Section \ref{sec:adj change when graph change}, and derive an algorithm to find valid adjustment sets when only the skeleton of a DAG is known in Section \ref{sec:experiment}.

\section{Preliminaries}\label{sec:definition}

\subsection{Definitions}
As this paper focuses on finding valid adjustment sets that are robust to graph misspecification, here we recall the definition of a valid adjustment set, as well as other graphical notions used to define our procedure.

A graph $\G$ consists of a set of vertices $\mathcal{V}$ and a set of \emph{directed} edges ($\to$)  $\mathcal{E}$ containing pairs of distinct vertices. For an edge in $\mathcal{E}$ connecting the vertices $a$ and $b$, we say that these two vertices are the \emph{endpoints} of the edge. The graphs in this paper are \emph{simple} (that is, there is at most one edge between any pair of vertices).

A \emph{path} of length $k$ is an alternating sequence of $k+1$ distinct vertices $v_{i}$ and edges connecting $v_i$ and $v_{i+1}$ for $i = 0,\dots, k-1$. A path is \emph{directed} if its edges are all directed and point from $v_i$ to $v_{i+1}$. A \emph{directed cycle} is a directed path of length $k \geq 1   $ plus the edge $v_k \rightarrow v_0$, and a graph $\G$ is \emph{acyclic} if it has no directed cycle. A \emph{graph} $\G$ is called an \emph{directed acyclic graph} (DAG) if it is \emph{acyclic}. 

For a vertex $v$ in a DAG $\G$, we define the following sets:
\begin{align*}
\pa_{\G}(v) &= \{w: w \rightarrow v \text{ in } \G\};\\
\an_{\G}(v) &= \{w: w \rightarrow \cdots \rightarrow v \text{ in } \G \text{ or } w=v\}; \\
\de_{\G}(v) &= \{w: v \rightarrow \cdots \rightarrow w \text{ in } \G \text{ or } w=v\}.
\end{align*}
They are referred to as the \emph{parents}, \emph{ancestors}, 
 \emph{descendants} of $v$, respectively. These operators are also defined disjunctively for a set of vertices $W \subseteq \mathcal{V}$. For example $\pa_{\G}(W) = \bigcup_{w \in W} \pa_{\G}(w)$. We sometimes ignore the subscript if the graph we refer to is clear, for example $\an (v)$ instead of $\an_{\G}(v)$.

A \emph{topological ordering} is an ordering on the vertices such that if $w \in \an_\G(v)$ then $w$ precedes $v$ in the ordering. There might be several topological orderings for any single graph.

To define the criterion for a set being a valid adjustment set, we need the following concept.

\begin{definition}
For a DAG $\G$, let 
 \begin{align*}
\cn(\G,x,y) &:= \{z: \exists \text{ a directed path in } \G \text{ from } x \text{ to } y \\
& \quad \quad \quad \quad \text{ and } z \text{ is on the path} ; z \neq x \} ;\\
\forb(\G,x,y) &:= \de(\cn(\G,x,y)) \cup X.
\end{align*}

\end{definition}

\begin{definition}\label{ref: valid adjustment set}
A set $A$ is a valid adjustment set in a DAG $\G$ if:
\begin{align}
    (i&) \quad A \cap \forb(\G,x,y) = 
    \emptyset,\label{criteria 1}\\
    (ii&) \quad \text{every non-causal path from } x \text{ to } y \text{ is blocked by } A. \label{criteria 2}
\end{align}    
\end{definition}

In \citet{van2019separators},  an equivalent definition to Definition \ref{ref: valid adjustment set} is proved, which we list here as Definition \ref{def:equiv def for adj set} and which will be used later for convenience.

\begin{definition}
For a DAG $\G$, let $\G^{pd}_{xy}$ be the DAG created by removing the directed edges of $x$ to any of its children, which are also ancestors of $y$.    
\end{definition}

\begin{definition}\label{def:equiv def for adj set}
A set $A$ is a valid adjustment set in a DAG $\G$ if:
\begin{align}
    (i&) \quad A \cap \forb(\G,x,y) = 
    \emptyset,\\
    (ii&) \quad x \text{ is d-separated from } y \text{ by } A \text{ in } \G^{pd}_{xy}. 
\end{align}  
\end{definition}

To facilitate the theory developed in Section \ref{Sec: Adjustment sets}, we define two more sets.
\begin{definition} 
For a DAG $\G$, let
 \begin{align*}
 \rch(\G,x,y) &:= \{z: \exists \text{ a path in the skeleton of } \G \text{ from } \\ & \quad \quad \quad \quad x \text{ to } y \text{ and } z \text{ is on the path}\} ;\\
 \clr(\G,x,y) &:= \{\text{all colliders on any path from } x \text{ to }  y \text{ on } \G_{\rch(\G,x,y)}\}.
\end{align*} 
\end{definition}

We sometimes do not write $x,y$ as they remain unchanged throughout the paper. For example, the above two sets can be written as $\rch(\G)$ and $\clr(\G)$.

Finally, we include the \emph{optimal} adjustment set, which was proven by \citet{henckel2022graphical} and \citet{smucler2022efficient} to be `optimal' in the sense that it is the most \emph{efficient} valid adjustment set in the linear Gaussian model and the non-parametric causal model, respectively. In a nutshell, this means that regressing on the optimal set leads to the lowest variance of the estimated treatment effect among all valid adjustment sets (for a complete definition of efficiency in this setting, see \citet{henckel2022graphical} and \citet{smucler2022efficient}). We will compare the adjustment sets found by our method to the optimal adjustment set in Section \ref{sec:listingefficient}.

\begin{definition}\label{def:optimaladjustmentset}
For a DAG $\G$, let the optimal adjustment set be
\begin{align*}
\opadj(\G,x,y) := \pa(\cn(\G,x,y)) \setminus \forb(\G,x,y).
\end{align*}
\end{definition}

\subsection{Bayesian Set-Up}\label{Sec:set up}

Adopting a Bayesian framework, let $\mathcal{D}$ be the data and $\mathcal{M}$ be a probabilistic model (prior) on the graphs. We assume that the skeleton is known, and take as prior a uniform distribution on the topological ordering over the set of vertices conditioned on  $x$ preceding $y$. Given a graph $\mathcal{G}$, the data are generated via a standard process, for example a Gaussian linear model described in Section \ref{sec:experiment}, which allows us to compute the BIC \citep{koller2009}.

Let $\mathcal{P}$ denote the power set over the vertex set $\mathcal{V}$. Furthermore, let $P_{\mathcal{P} \mid \mathcal{G}, \mathcal{M}}$ be a vector of probabilities on subsets of variables where for each entry, denoted by $ P_{A}$, is the posterior probability of the set $A$ being a valid adjustment set given the data and prior on the graphs. This probability can be computed as follows:
$$
P_{A} = \sum_{G} I_{(A \text{ valid in }G)} P(G \mid \mathcal{D}, \mathcal{M}).
$$

The indicator function can be checked using graphical criteria in $O(n)$ time \citep{van2019separators}, and the latter posterior probability can be approximated via sampling. 

What we are interested in is the following. Given any set $A$, one can approximate its posterior probability  of being valid (relatively) quickly by sampling. If we do not have a specific target set or if we wish to find more sets that are likely to be valid, a naive approach would be to try to compute the whole vector $P_{\mathcal{P} \mid \mathcal{G}, \mathcal{M}}$. This is certainly inefficient and incomputable when the number of nodes grows. 

This paper shows that when the sampled graph $\G$ is changed to $\G'$ by some local change on a node, the indicator function can be checked faster under certain conditions. In addition, we provide techniques to narrow down the search range, avoiding computing the entire vector while maintaining the accuracy and efficiency of the adjustment sets.

\subsection{Why not the bootstrap?}\label{sec: no boostrap }

An obvious question arises as to whether we should apply DAG learning algorithms to bootstrapped data sets. Our approach has two advantages over bootstrapping: computational efficiency and uncertainty quantification.

First, checking validity of sets in the DAGs returned by bootstrap is less efficient than our method as one needs to examine each set on each DAG separately and independently while our sampling technique allows us to quickly update the validity of sets.

Second, a Bayesian method gives us a useful measure of uncertainty: the marginal distribution of the probability of each set being valid. This measure of uncertainty can, in future work, be propagated to the final outcome analysis so that it will reflect the full uncertainty of the analysis.

\section{Adjustment sets when the graph changes}\label{Sec: Adjustment sets} \label{sec:adj change when graph change}

We present the main results on how a local change on a node in a graph affects the validity of some (or all) adjustment sets.  The results from this section lead to a reduction in the number of sets that need to be checked for validity, as discussed in more detail in Section \ref{sec:experiment}. An additional result, not directly needed for our algorithm, is presented as Proposition \ref{prop:precedingtopological} in the Appendix.

\subsection{Insignificant nodes }\label{sec:insignificant nodes}

In this section, we will show that for certain nodes, if only their neighbours are changed when we change or alter its position in the topological ordering, then any valid adjustment set remains valid after such a change. We can thus obtain the validity of adjustment sets when these nodes appear to be the changed node. An example of a change of the position of a node in the topological ordering would be as follows. Suppose that the topological ordering is $a,b,c,d,x,e,y$ and we alter one vertex, say $b$, then one possible new ordering would be $a,c,d,x,b,e,y$.

\begin{proposition}\label{prop:invariant adjustent set}
    Consider two DAGs $\G$ and $\G'$ such that in topological orderings only one node $z$ is altered. If the following are the same in both graphs: $$\G_{\rch(\G)}=\G^{'}_{\rch(\G^{'})}, \clr(\G)=\clr(\G^{'}), \forb(\G) = \forb(\G^{'}),$$ and 
    $$\de_{\G}(c) = \de_{\G^{'}}(c) \text{ for every c} \in \clr(\G) , $$then $A$ is a valid adjustment set in $\G$ if and only if it is valid in $\G'$.
\end{proposition}

\begin{proof}
    Note that $\G^{pd}_{xy} = \G^{'pd}_{xy}$. WLOG, let $A$ be a valid adjustment set in $\G$. Suppose that there is a d-connecting path from $x$ to $y$ in $\G^{'pd}_{xy}$, given $A$. Now, the path is the same in $\G'$. The only mechanism through which it could become open is if some colliders $C$ become open when altering $z$, which means $A \cap \de_{\G'}(C) \neq \emptyset$ and $A \cap \de_{\G}(C) = \emptyset$, which is a contradiction to the last condition.
\end{proof}

\begin{lemma}\label{lemma: good node}
 Consider the setting in Proposition \ref{prop:invariant adjustent set}. Suppose $z$ is a node not in 
 \begin{equation} \label{eq:conditioninsignificant}
 \rch(\G) \cup \forb(\G) \cup \de(\clr(\G)) \cup \pa(\forb(\G) \cup \de(\clr(\G))).
 \end{equation}
 Then the conditions in Propositions \ref{prop:invariant adjustent set} are satisfied.   
\end{lemma}

\begin{proof}
Trivially, $\rch(\G) = \rch(\G')$. Then we show that $$
    \forb(\G) = \forb(\G').
    $$
    Note that any node on any directed path from $x$ to $y$ is in $\rch(\G)$, so any directed path from $x$ to $y$ is preserved. For similar reasons, there is no new directed path in $\G'$. Now, forbidding $z$ in $\forb(\G) \cup \pa(\forb(\G))$ proves the equality.    

    It is clear that $\clr(\G)$ remains the same because $z \notin \rch(\G)$ so $\G_{\rch(\G,x,y)} = \G'_{\rch(\G,x,y)}$. Suppose $\de_{\G}(c)  \neq \de_{\G'}(c)$ for some node $c$ in $\clr(\G)$. Consider any directed path from $c$ to any node in $\de_{\G'}(c)$ but not $\de_{\G}(c)$. Let $d$ be the first node that is not in $\de_{\G}(c)$. That means that the edge connecting $d$ is altered, so $z$ is the end point of this edge. But then $z$ is in $\de(\clr(\G)) \cup \pa(\de(\clr(\G)))$.   
\end{proof}

The condition in Lemma \ref{lemma: good node}
 is strictly weaker than those in Proposition \ref{prop:invariant adjustent set} as one can easily construct an example where $z \in \de(\clr(\G))$ but $\de_{\G}(c)$ remain unchanged for every $c \in \clr(G)$.

\begin{corollary}\label{cor:separatesets}
Let $z$ be an altered node in $\G$. Then if $z$ is not in :
\begin{enumerate}
    \item $\rch(\G)$, then $\G_{\rch(\G)}$ and $\clr(\G)$ remain unchanged.
    \item $\rch(\G) \cup \forb(\G) \cup \pa(\forb(\G))$, then  $\forb(\G)$ remain unchanged.
    \item $\rch(\G) \cup \de(\clr(\G)) \cup \pa(\de(\clr(\G)))$, then $\de_\G(c)$ remains unchanged for every $c\in\clr(\G)$.
\end{enumerate}
\end{corollary}

\subsection{Targeted adjustment sets}\label{sec:targeted sets}

As we stated in Corollary \ref{cor:separatesets}, the set \eqref{eq:conditioninsignificant} in Lemma \ref{lemma: good node} can also be considered separately, each leading to different sets remaining unchanged. The utility is that if only a part of the condition is satisfied, then we can still have partial results. We refer to the unchanged sets from Corollary \ref{cor:separatesets} as `targeted' adjustment sets. 


\begin{lemma}\label{lemma:good nodes}
    Suppose the altered node $z$ is not in $\rch(\G)$. If $$A \cap ((\forb(\G) \cup \forb(\G')) \setminus (\forb(\G) \cap \forb(\G'))) = \emptyset$$ and $$A \cap (\bigcup_{c \in \clr(\G)} ((\de_{\G}(c) \cup \de_{\G'}(c))) \setminus (\de_{\G}(c) \cap \de_{\G'}(c))) = \emptyset,$$ then $A$ is a valid adjustment set in $\G$ if and only if it is valid in $\G'$.
\end{lemma}

\begin{proof}
    Suppose $A$ is a valid adjustment set in $\G$. Because $A \cap (\forb(\G) \setminus \forb(\G')) = \emptyset$ and $A \cap (\forb(\G') \setminus \forb(\G)) = \emptyset$, we have $A \cap \forb(\G') = \emptyset$.

    As the altered node $z$ is not in $\rch(\G)$, no directed path from $x$ to $y$ is changed and there is no new directed path. Therefore, only the descendant of the nodes on the directed paths could change. Therefore, note that $\G^{pd}_{xy} = \G^{'pd}_{xy}$. If $A$ is not a valid adjustment set in $\G'$, then there is a path that is open in $\G^{pd}_{xy}$ but not in $\G^{'pd}_{xy}$. Consider any path between $x$ and $y$ that is open in $\G^{pd}_{xy}$ but not in $\G^{'pd}_{xy}$. The only reason why the connectivity changes is that some colliders $c$ are conditioned in $\G^{pd}_{xy}$ but not in $\G^{'pd}_{xy}$. Therefore, $A \cap \de_{\G}(c) \neq \emptyset$ and $A \cap \de_{\G'}(c) = \emptyset$, but then this contradicts the given conditions. 
\end{proof}

If $\forb(\G)$ is unchanged after reordering $z$, then for any $A$,  $A \cap ((\forb(\G) \cup \forb(\G')) \setminus (\forb(\G) \cap \forb(\G'))) = \emptyset$, and we only need to check those sets that contain the changed descendants of the colliders. Similarly, if the descendants of colliders stay the same, it is sufficient to check those sets that contain the changed forbidden nodes. If both of them are unchanged, the conditions in Proposition \ref{prop:invariant adjustent set} are satisfied.






\subsection{Listing efficient valid adjustment sets}\label{sec:listingefficient}

Recall the definition of the optimal adjustment set $\opadj(\G,x,y)$ in Definition \ref{def:optimaladjustmentset}. Essentially, Proposition \ref{prop:efficient adjustment set} says that we can push valid adjustment sets towards the optimal adjustment set and as long as it remains a valid adjustment set, it is a more efficient adjustment set. 

\begin{proposition}\label{prop:efficient adjustment set}
Suppose $A$ is a valid adjustment set, then for any $A_O \subseteq \opadj(\G,x,y)$ such that $A_O$ is a valid adjustment set and $(\de(A) \cap \opadj(\G,x,y)) \subseteq A_O$, $A_{O}$ is a more efficient valid adjustment set than $A$.
\end{proposition}
\begin{proof}
 Consider any $A_O$ that is a valid adjustment set. By Theorem 3.4 in \citet{henckel2022graphical}, it is sufficient to show that $y \indep A \setminus A_{O} \mid \{x\} \cup A_{O}$ and $x \indep A_{O} \setminus A \mid A$.  
    
To prove $x \indep A_{O} \setminus A \mid A$. Let $z \in A_{O} \setminus A$. Suppose that there exists a connecting path $p$ between $x$ and $z \in A_{O} \setminus A$ given A. The aim here is to construct a connecting non-causal path between $x$ and $y$ given $A$.

This path $p$ is not directed because $z \in \opadj(\G,x,y)$ is not a forbidden node. Since $z \in \opadj(\G,x,y)$, there exists a directed path from $z$ to $y$ and all nodes in the middle (if exists) are forbidden nodes which, including $z$ are then not in $A$ and therefore if we consider the joint path by the connecting path $p$ between $x$ and $z$, and the directed path from $z$ to $y$, this is a non-causal connecting path between $x$ and $y$ given $A$, contradicting the assumption that $A$ is a valid adjustment set.

To prove $y \indep A \setminus A_{O} \mid \{x\} \cup A_{O}$. Let $z \in A \setminus A_{O}$. Clearly, $z \notin \opadj(\G,x,y)$. Suppose that there exists a connecting path $p$ between $y$ and $z$ given $A_{O} \cup \{x\}$. Suppose that this path does not include $x$. This path cannot be directed. If it is directed from $y$ to $z$, then $z$ is a forbidden node. If it is directed from $z$ to $y$, then since $z$ is not a forbidden node and $y$ is on the end of the path, there is at least one node on the path, including $z$ itself, such that it is in $\opadj(\G,x,y))$. Then this path is blocked by conditioning on $A_{O}$. 

Now consider the last edge of the path $p$, which has $y$ as an endpoint. If it is going out of $y$, then there must be a collider which blocked as it and its descendants are forbidden nodes. If it is going into $y$, suppose that it is $q_1 \rightarrow y$. If $q_1$ is in $\opadj(\G,x,y))$ then we arrive at a contradiction. As we have shown that the path $p$ cannot be directed, there must be a node $q_{i+1}$ and $q_i, \cdot, q_1$ such that $q_{i+1} \leftarrow q_i \rightarrow ,\cdots, \rightarrow q_1 \rightarrow y$. If any of $q_1, \cdots, q_i$ is in $\opadj(\G,x,y))$, then we are done. Suppose not. in particular $q_i$ is not $\opadj(\G,x,y))$ and therefore it is on some causal path from $x$ to $y$. Consider the position of $z$, which is more left of $q_i$. If the part from $z$ to $q_i$ is directed from $q_i$ to $z$, then by definition $z$ is a forbidden node, which is a contradiction. If it is not a directed path, then there must be a collider, which, including its descendants is a forbidden node and hence this path $p$ is not open.

Suppose that any such path $p$ includes $x$. Then $x$ must be a collider in order for $p$ to be open. So the subpath $p'$ from $x$ to $y$ is non-causal. Again we consider the last edge involving $y$. It cannot going out of $y$ because if so, there must be a collider otherwise contradicting the assumption that $x$ proceeds $y$. This collider and its descendant are then forbidden nodes and also does not include $x$, so the path $p$ could not be open given $A_{O}$ and $x$. 

If there exist colliders on the subpath $p'$ that have $x$ as a descendant. Consider the collider that is closest to $y$ and the new path $p''$ by joint the directed path from the collider to $x$ and the subpath between $y$ and the collider. If there doesn't exist any such collider, then let $p''=p'$. 

Suppose that the last edge of $p$ is $q_1 \rightarrow y$. Consider the last node $q_{i+1}$ from $y$ on the path $p$ such that $q_{i+1} \leftarrow q_i \rightarrow ,\cdots, \rightarrow q_1 \rightarrow y$. Note that all these $q's$ is also on $p'$. By similar argument, if any of $q_1, \cdots, q_i$ is in $\opadj(\G,x,y))$, then we are done. Suppose not. In particular $q_i$ is not in $\opadj(\G,x,y))$ and therefore is on some causal path from $x$ to $y$. We consider any of the subpath of $p'$ and $p''$ from $x$ to $q_{i}$, which both includes $q_{i+1}$. If the path is directed, then it is a contradiction because the graph is acyclic. If there is a collider then it and its descendants are therefore forbidden nodes and so $p$ cannot be open given $A_{O} \cup \{x\}$.
\end{proof}

Proposition \ref{prop:efficient adjustment set} and its proof are analogous to Theorem 3.9 in \citet{henckel2022graphical}.

For any valid adjustment set $A$, there must exist at least one such $A_O$ as $\opadj(\G,x,y)$ clearly satisfies the conditions in Proposition \ref{prop:efficient adjustment set}.

 The result of Proposition \ref{prop:efficient adjustment set} means that we can list efficient valid adjustment sets by going through subsets of the optimal adjustment sets, which is a local procedure rather than going through all subsets of variables.

 Example \ref{example:pushing} shows that for a valid adjustment set $A$, merely pushing it towards its descendant in the optimal adjustment may result in an invalid adjustment set.

 \begin{example}\label{example:pushing}
  Consider the DAG $\G$ in Figure \ref{fig: invalid adjustment set by pushing towards descendants}.  $\opadj(\G,x,y) = \{4,8,9\}$ and let $A = \{6,7\}$. Note that $\de(A) \cap \opadj(\G,x,y) = \{4\}$, which is an invalid adjustment as the path $x \leftarrow 6 \rightarrow 5 \leftarrow 7 \leftarrow 8 \rightarrow y$ is open and non-causal between $x$ and $y$. We need an additional node $\{8\}$ such that $\{4,8\}$ fully blocks all the non-causal paths between $x$ and $y$ and is more efficient than $A$.
  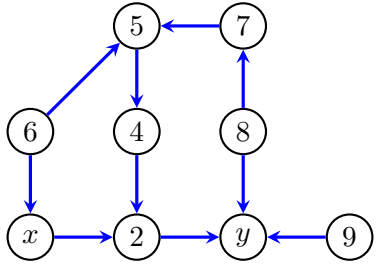
\begin{figure}
\centering
  \begin{tikzpicture}
  [rv/.style={circle, draw, thick, minimum size=6mm, inner sep=0.8mm}, node distance=14mm, >=stealth]
  \pgfsetarrows{latex-latex};
\begin{scope}
  \node[rv]  (1)            {$x$};
  \node[rv, right of=1] (2) {$2$};
  \node[rv, right of=2] (3) {$y$};
  \node[rv, above of=2] (4) {$4$};
  \node[rv, above of=4] (5) {$5$};
  \node[rv, above of=1] (6) {$6$};
  \node[rv, right of=5] (7) {$7$};
  \node[rv, right of=4] (8) {$8$};
  \node[rv, right of=3] (9) {$9$};
  
  \draw[->, very thick, blue] (1) -- (2);
  \draw[->, very thick, blue] (2) -- (3);
  \draw[->, very thick, blue] (4) -- (2);
  \draw[->, very thick, blue] (6) -- (1);
  \draw[->, very thick, blue] (6) -- (5);
  \draw[->, very thick, blue] (5) -- (4);
  \draw[->, very thick, blue] (7) -- (5);
  \draw[->, very thick, blue] (8) -- (7);
  \draw[->, very thick, blue] (8) -- (3);
  \draw[->, very thick, blue] (9) -- (3);
  \end{scope}
\end{tikzpicture}
\caption{Invalid adjustment set example }
\label{fig: invalid adjustment set by pushing towards descendants}
\end{figure}   
 \end{example}

\section{Algorithmic considerations}\label{sec:experiment}
\subsection{Main graphical assumption}

In this paper, we assume that the skeleton of the true DAG is known, so we are sampling directions of edges. There are several reasons to make this assumption, and we here further justify why this assumption is not too impractical. 

From a theoretical point of view, recovering the skeleton is a simpler problem and easier to approach and solve. Indeed, as we have shown, several results on invariance of valid adjustment sets when changing graphs can be derived rather directly. These results then bring us some insight on which types of graph changes would alter the validity of certain sets.

From a practical point of view, the errors of graph-learning algorithms often come from the wrong orientation of the edges rather than the wrong prediction of edges. Building a skeleton is easier than orienting the edge \citep{tsamardinos2006max,mwebaze2010fast}, which can depend on finding an association between variables. Also, in some constrain-based algorithms, the edge orientation procedure follows after learning the skeleton, so inherently edge orientation admits more mistakes potentially. In addition, in applications, experts could have ideas about which variables are directly related but are uncertain about the causal direction. 

There has been a lot of work on sampling DAGs, including partitioning MCMC by \citet{kuipers2017partition,sample-1,sample-2} . In this work, we will sample DAGs by sampling the topological orderings. By assuming a known and fixed skeleton, each topological ordering corresponds to a DAG. The way we sample the topological ordering is as follows:
\begin{itemize}
    \item[(i)] randomly select a node
    \item[(ii)] randomly re-position the selected node and ensure that $x$ still precedes $y$; if not, repeat the procedure.
\end{itemize}

There are several reasons why we choose this method. Skeletons are often easier to detect, as we explained earlier. Also by sampling topological orderings, it naturally corresponds to a DAG while if we sample by inverting edges, it may result in cyclic graphs, and thus we avoid the computational costs to check whether the graph is acyclic. Moreover, changing the position of a node in the topological ordering only changes the graph locally. That is, only the neighbours of the node (its parents and children) are changed, and it is therefore easier to check the validity of adjustment sets in certain situations, as we have seen in Section \ref{Sec: Adjustment sets}.

By changing the ordering of a vertex, we are flipping some collider/non-collider triple that involves the vertex. This is like the turning stage introduced by GIES \citep{hauser2012characterization} and in fact, if the given skeleton is true, this has been proved to recover the true Markov equivalence class consistently in the limit of infinite data size \citep{linusson2022edges,linusson2023greedy}.

Let $\tau$ denote topological ordering and $K$ be a given skeleton. We consider a Markov chain with a stationary distribution proportional to the true ordering $P(\tau \mid \mathcal{D},K)$ (we sometimes ignore $K$ as it is fixed), which can be produced by a Metropolis Hastings algorithm. The acceptance probability with a newly proposed topological ordering $\tau'$ is:
\begin{equation}
    p = \min\{1,\frac{q(\tau\mid \tau') P(\tau' \mid \mathcal{D},K)}{q(\tau'\mid \tau) P(\tau \mid \mathcal{D},K)}\}.
\end{equation}
Note that $P(\tau \mid \mathcal{D},K) \propto P(\mathcal{D}\mid \tau,K) = P(\mathcal{D} \mid \G)$, where $\G$ is the DAG with skeleton $K$ and topological ordering $\tau$. The last quantity is the likelihood of empirical data given a graph and can be estimated by BIC \citep{andrews2018scoring,kitson2023survey}.


\subsection{Algorithm}\label{sec:algorithm}

Code is available at: https://github.com/zhongyi960403.

Our main algorithm, Algorithm \ref{alg:get_adjust} uses the theory from Section \ref{sec:adj change when graph change} to find valid adjustment sets when only the skeleton of the underlying DAG is known.

Some remarks about the algorithm: 
\begin{itemize}
    \item[i] By computing $L_{cur}$ from $\G_{cur}$, we mean that by utilizing Proposition \ref{prop:efficient adjustment set}, we test the validity of subsets of the optimal adjustment set of $\G_{cur}$. If an adjustment set is valid, we would include it in $L_{cur}$.  We usually test those subsets with size only one less than $\opadj(\G_{cur},x,y)$ and enough to show effectiveness.
    \item[ii] If the altered node satisfied the results in Sections \ref{sec:insignificant nodes} and \ref{sec:targeted sets}, we would update $L_{cur}$ faster. For example, we can copy the previous $L_{cur}$ if Lemma \ref{lemma:good nodes} is satisfied.
    \item[iii] Converting $L$  to $L_{\vec{s}}$ by $L_{en},\vec{s}$. Dividing $L$ by $L_{en}$ gives the empirical probability of being valid. The vector $\vec{s}$ is a vector of real numbers from 0 to 1. For a value $s$ in $\vec{s}$, we would include all subsets in $L_{en}$ that has higher probability than the highest value in $L_{en}$ timed by $s$. The higher the confidence $s$, the higher probability the sets appear and thus are more likely to be a valid adjustment set.
\end{itemize}
\begin{algorithm}
    \caption{Get\textunderscore adjustment\textunderscore set} 
    \label{alg:get_adjust} 
    \KwIn{Data $D$, x, y, skeleton $S$, sampling number $M$,threshold vector $\vec{s}$}
    \KwOut{A list of adjustment sets $L$}
    \SetAlgoLined  
    \textbf{Initialize} for $\G_{cur}, S_{cur}, L$\;
    \text{Learn }$\G_{cur}$ \text{ from }GES with $S$; Let its $BIC$ be $S_{cur}$\;
    \text{Define } $L,L_{cur}$ \text{ as a sparse vector, where}\,
    \text{each entry is a subset and its value is frequency of being valid}\;
    \text{Define } $L_{en}$ \text{ as a sparse vector, where}\,
    \text{each entry is a subset and its value is frequency of being tested}\;
    \text{Compute } $L_{cur}$ \text{ from } $\G_{cur}$\;
    \For{$m \gets 1$ \KwTo \(M\)}{
        \text{Sample } $\G_{new}$ \text{ from } $\G_{cur}$ \text{ and compute } $S_{new}$\;
        $p=\min(1,\exp(S_{new}-S_{cur}))$\;
        \If{$\uniform(0,1) < p$}{
          $\G_{cur} = \G_{new}$;$S_{cur} = S_{new}$\;
          \text{Compute } $L_{cur}$ \text{ from } $\G_{cur}$\;
          \text{Update } $L_{en}$\;
        }
        \Else{
          \text{next}\;
        }
        $L = L+L_{cur}$\;
    }
    \text{Convert } $L$ to $L_{\vec{s}}$ by $ L_{en},\vec{s}$\;
    \Return $L_{\vec{s}}$
\end{algorithm}

We use the output from GES as the initial DAG, but one can also start with any random DAG with the given skeleton.

In practice, a large sample size may cause the sampling to get stuck. One possible solution is to reduce the size of the data and have multiple runs to find the most common set. See the application example in Section \ref{Sec: real world example} for details. Another option would be to multiply the acceptance probability by a given constant, for example, according to size of data, which is implemented in our code.

\subsection{Complexity bound and consistency}\label{sec:consistency}

In this Section, we provide some theoretical analysis for Algorithm \ref{alg:get_adjust}.

\begin{proposition}
    The complexity of Algorithm \ref{alg:get_adjust} is of $O(n^2)$ for each sampled DAG.
\end{proposition}
\begin{proof}
The time to verify validity is linear for each set \citep{van2019separators}. For each sampled DAG, we check the $O(n)$ adjustment sets as the optimal adjustment sets are at most $O(n)$ and we only consider subsets that have a size one less than it.
\end{proof}

Now under a frequentist framework and assuming that there is one true underlying DAG, we analyse the behaviour of our algorithm when the data $\mathcal{D}$ are Gaussian or discrete and have infinite size. To do this, we need the concept of `completely partially oriented' DAG (CPDAG) \citep{Spirtes}. 

In short, CPDAG is a method for representing the Markov equivalence class of a DAG. It has the same skeleton as the DAG, consisting of directed and undirected edges. An edge in a CPDAG is directed if and only if all DAGs that are Markov equivalent to the given DAG have this edge. An edge is undirected if and only if there are two Markov equivalent DAGs where this edge is oriented oppositely.

The concept of `CPDAG' is important when we can only identify up to Markov equivalent class of DAGs, which are the cases for Gaussian linear models and discrete models. 

When data $\mathcal{D}$ are Gaussian or discrete and have infinite size, a DAG/CPDAG learning algorithm, for example, GES \citep{chickering2002optimal}, will identify the correct CPDAG. In our case, the algorithm will converge to the Markov equivalence class (MEC) of the true underlying DAGs. 

\begin{proposition}\label{prop: infinite data, true CPDAG}
Suppose the following:
\begin{itemize}
    \item the given skeleton $S$ is correct;
    \item the data $\mathcal{D}$ are Gaussian or discrete and have infinite size;
\end{itemize}

then as the sampling time $M \to \infty $, the sampled DAG is always Markov equivalent to the true underlying DAG.
\end{proposition}
\begin{proof}
This can be proven by the results of \citet{linusson2022edges}. When the data size goes to infinity, the acceptance probability is 0 if the new DAG has worse BIC and 1 otherwise. By the newly proven local consistency of BIC and the consistency of Skeleton Greedy CIM in \citet{linusson2022edges} (Propositions 4.2 and 4.3), there is always a node that can be altered to improve BIC unless we are already in the MEC of the true underlying DAG.
\end{proof}

However, identifying the true MEC does not necessarily mean that we can identify at least one valid adjustment set. See the following example.

\begin{example}
  \begin{figure}
\centering
  \begin{tikzpicture}
  [rv/.style={circle, draw, thick, minimum size=6mm, inner sep=0.8mm}, node distance=14mm, >=stealth]
  \pgfsetarrows{latex-latex};
\begin{scope}[xshift=-3cm]
  \node[rv]  (1)            {$x$};
  \node[rv, above of=1] (2) {$1$};
  \node[rv, right of=1] (3) {$y$};
  \node[ below of=1,xshift=0.7cm,yshift=0.5cm]{(i)};
  
  \draw[->, very thick, blue] (1) -- (3);
  \draw[->, very thick, blue] (2) -- (3);
  \draw[->, very thick, blue] (2) -- (1);
  \end{scope}
  \begin{scope}[xshift=0cm]
  \node[rv]  (1)            {$x$};
  \node[rv, above of=1] (2) {$1$};
  \node[rv, right of=1] (3) {$y$};
   \node[ below of=1,xshift=0.7cm,yshift=0.5cm]{(ii)};
  \draw[->, very thick, blue] (1) -- (3);
  \draw[->, very thick, blue] (2) -- (3);
  \draw[->, very thick, blue] (1) -- (2);
  \end{scope}
  \begin{scope}[xshift=3cm]
  \node[rv]  (1)            {$x$};
  \node[rv, above of=1] (2) {$1$};
  \node[rv, right of=1] (3) {$y$};
   \node[ below of=1,xshift=0.7cm,yshift=0.5cm]{(iii)};
  \draw[->, very thick, blue] (1) -- (3);
  \draw[->, very thick, blue] (3) -- (2);
  \draw[->, very thick, blue] (1) -- (2);
  \end{scope}
\end{tikzpicture}
\caption{True MEC is not enough}
\label{fig: True MEC is not enough}
\end{figure}
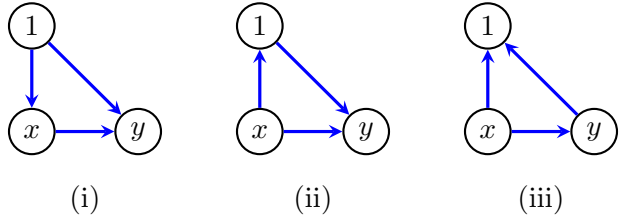     
Consider Figure \ref{fig: True MEC is not enough}. Suppose that the true underlying DAG is $\G$ in (i). Both graphs $\G'$ and $\G''$ in (ii) and (iii) are Markov equivalent to it and satisfy the condition that $x$ precedes $y$. The only valid adjustment set is $\{1\}$, while $\emptyset$ is valid in both $\G'$ and $\G''$. Therefore, by a uniform sampling of topological ordering conditioned on $x$ preceding $y$, we would obtain a probability of one third for $\{1\}$ and a probability of two thirds for $\emptyset$, without identifying the true valid adjustment set.

This is an issue related to the concept of `amenability' that arises when studying the graphical criterion for valid adjustment set in CPDAGs \citep{perkovi2018complete}. 

\begin{definition}
    A CPDAG $\G$ is amenable if for every possibly directed path from $x$ to $y$, it begins with $x \rightarrow$.
\end{definition}

A CPDAG must be amenable to have a valid adjustment set. The CPDAG for $\G$ in this example is not amenable, so there is no valid adjustment set in the CPDAG, that is, there is no adjustment set that is valid for all DAGs lying in the MEC represented by the CPDAG.
\end{example}

However, we can show that when the underlying DAG has an amenable CPDAG, our algorithm will pick only valid adjustment sets.

\begin{proposition}\label{prop: convergence result}
Suppose the following:
\begin{itemize}
    \item the given skeleton $S$ is correct;
    \item the data $\mathcal{D}$ are Gaussian or discrete and have infinite size;
    \item the CPDAG of the true underlying DAG $\G$ is amenable;
\end{itemize}

then as the sampling time $M \to \infty $, any adjustment set whose frequency of being valid divided by the testing frequency converge to 1 only if this set is a valid adjustment set in $\G$. 

In addition, the optimal adjustment set is included in the output list of sets in Algorithm 1.   
\end{proposition}
\begin{proof}
By Proposition \ref{prop: infinite data, true CPDAG}, Algorithm 1 converges to DAGs in the CPDAG of the true DAG. Because the CPDAG is amenable, there exists an optimal adjustment set which is also the optimal adjustment set for every DAG represented by the CPDAG (Lemma E.7 in \citet{henckel2022graphical}). Hence, at every iteration, this optimal adjustment set is picked up and will be included in the output list of sets. 

If, for a set $A$, its frequency of being valid divided by the testing frequency converges to 1, this means that once we arrive at the true CPDAG, this set is valid for every DAG, as the number of Markov equivalent DAGs is finite for the CPDAG. In particular, it is valid in the true underlying DAG.
\end{proof}

In Section \ref{sec:experiment}, we conduct experiments about how `amenability affects algorithms' performance.

\begin{remark}
The results in this section merely serve as a sanity check to ensure that our algorithm produces meaningful outputs as the number of data points and the sampling time approach infinity. They are not of any practical use, since if the data are infinite in size, any DAG/CPDAG learning algorithm can identify a valid adjustment set. The essence of our algorithms lies in the situation where the data are finite and empirical mistakes are made.
\end{remark}

\section{Experiments and real data analysis}
In this section, we conduct three experiments for (1) comparison with the baseline algorithm; (2) how the amenability of ground truth graphs affects performance; (3) how sample size and sampling time affect algorithm performance conditioning on whether the ground truth graphs are amenable or not.

\subsection{Simulation set-up and measure metrics}

For Experiment 1, for each even integer from 6 to 20, we simulate 100 DAGs, which is done by generating a random skeleton with expected edge degree 3, then randomly choosing a topological ordering to orient the skeleton. For each edge, the coefficient is uniformly distributed independently from -1 to 1. Then a random pair of $x$ and $y$, conditioned on $x$ being an ancestor of $y$, is chosen for each ground truth DAG. For each DAG, we generate 100 data points and put the required elements in Algorithm \ref{alg:get_adjust}. For each threshold $s$, we compute the corresponding precision and MSE. The sampling number is set to 100.

For Experiments 2, the settings are the same as those of Experiment 1 except that when we randomly generate DAGs, we require their CPDAGs to be amenable/non-amenable.

For Experiments 3, there are two runs and plots. For one of them, the sample sizes $\log(N) $ range from 3 to 10 with step 0.5 and the sampling time $M$ is set at 40. For the other, the sampling time $M$ varies from 1 to 30 with step 2 and the sample size is fixed at 100. The number of nodes is fixed to 10.

The metrics we use to evaluate algorithms are precision and mean squared error (MSE). Specifically, precision is the percentage of algorithm's output sets that are actually valid in the ground-truth DAGs. The MSE is evaluated with respect to the true causal effect and, if there are multiple adjustment sets for one data set, we average the MSE.

\subsection{Experiment 1: comparison to baseline with known/estimated skeleton}
Experiment 1, whose results are presented in  Tables \ref{tab:precision_true_vs_ges} and \ref{tab:mse_true_vs_ges}, serves to compare the performance between our algorithm and the baseline algorithms. The experiments are also conducted with both the true skeleton and the estimated skeleton to see how the assumption of a known skeleton affects precision to select valid adjustment sets. 

In the tables, each entry is in the form of $x/y$, corresponding to skel$_{\text{true}}$/skel$_{\text{GES}}$, where  skel$_{\text{true}}$ stands for the case where the skeletons are the skeletons of ground-truth DAGs, and skel$_{\text{GES}}$ represents the case where the skeletons are estimated by GES. The column $\opadj(\G_{g},x,y)$ corresponds to the optimal adjustment sets from the GES output. $\opadj(\G_{r},x,y)$ are the optimal adjustment sets from random DAGs generated from given true skeletons. The last column, MEC, is the percentage of GES output that reaches the true Markov equivalence class.

Selecting $s=1$ means that we only include the set(s) with the highest frequency to appear. Selecting $s=0$ means that we include all sets that appear to be valid in some sampled DAGs. In general, as $s$ decreases, the precision/MSE becomes worse. The numbers in $s=1$ do not necessarily beat the numbers in $s=0$ in finite sample sizes and sampling times.

The cases for $s=1/s=0.8$ with the given skeleton in general have the best performance. Our algorithms with estimated skeleton are generally worse than those with true skeletons, but there is still obvious improvement compared to baseline.

We also plot the computational time against the number of nodes per 100 DAGs in Figure \ref{fig: time plot}. Empirically, one can clearly see a linear tendency, showing that our algorithm scales well. Moreover, this justifies the theoretical analysis of algorithm complexity in Section \ref{sec:consistency}.

\begin{table}[htbp]
    \centering
    \caption{$x/y$: given skeletons/GES skeletons. $\opadj(\G_{g},x,y)$,$\opadj(\G_{r},x,y)$ are optimal adjustment sets from GES and random DAGs.MEC: percentage of true MEC.}
    \label{tab:precision_true_vs_ges}
    \sisetup{round-mode=places, round-precision=2}

    \begin{tabular}{r r r r r r r r r}
        \toprule
        $n$ & \multicolumn{8}{c}{Precision for skel$_{\text{true}}$ / skel$_{\text{GES}}$} \\
        \cmidrule(lr){2-9}
            & $s=1.0$ & $s=0.8$ & $s=0.5$ & $s=0.3$ & $s=0$ & $\opadj(\G_{g},x,y)$  & $\opadj(\G_{r},x,y)$ & MEC\\
        \midrule
         6  & \textbf{0.73}/0.60 & 0.70/0.58 & 0.63/0.54 & 0.55/0.45 & 0.38/0.31 & 0.41/0.48 & 0.18 & 0.38\\
         8  & 0.66/0.64 & \textbf{0.68}/0.63 & 0.65/0.58 & 0.55/0.49 & 0.37/0.28 & 0.42/0.50 & 0.33 & 0.22\\
        10  & \textbf{0.71}/0.65 & 0.71/0.64 & 0.70/0.61 & 0.62/0.57 & 0.44/0.33 & 0.50/0.49 & 0.29 & 0.19\\
        12  & \textbf{0.81}/0.72 & 0.80/0.70 & 0.80/0.69 & 0.75/0.67 & 0.46/0.35 & 0.47/0.49 & 0.22 & 0.12\\
        14  & \textbf{0.74}/0.64 & 0.73/0.64 & 0.70/0.62 & 0.67/0.59 & 0.40/0.29 & 0.48/0.51 & 0.25 &0.11\\
        16  & 0.68/0.62 & \textbf{0.69}/0.61 & 0.68/0.58 & 0.65/0.55 & 0.39/0.31 & 0.52/0.51 & 0.13 & 0.06\\
        18  & \textbf{0.75}/0.70 & 0.75/0.70 & 0.76/0.68 & 0.74/0.67 & 0.44/0.35 & 0.51/0.55 & 0.27 &0.12\\
        20  & 0.77/0.62 & \textbf{0.79}/0.62 & 0.76/0.63 & 0.74/0.64 & 0.41/0.35 & 0.50/0.47 & 0.34 & 0.06\\
        \bottomrule
    \end{tabular}
\end{table}

\begin{table}[htbp]
    \centering
    \caption{MSE comparison: skel$_{\text{true}}$ / skel$_{\text{GES}}$.}
    \label{tab:mse_true_vs_ges}
    \sisetup{round-mode=places, round-precision=2}

    \begin{tabular}{r r r r r r r r r}
        \toprule
        $n$ & \multicolumn{7}{c}{MSE$\times 10^{-2}$ for skel$_{\text{true}}$ / skel$_{\text{GES}}$} \\
        \cmidrule(lr){2-8}
         & $s=1$ & $s=0.8$ & $s=0.5$ & $s=0.3$ & $s=0$ & $\opadj(\G_{g},x,y)$ & $\opadj(\G_{r},x,y)$\\
        \midrule
         6  & \textbf{1.50}/1.75 & 1.57/1.87 & 1.74/2.02 & 1.87/2.11 & 2.61/2.70 & 2.31/2.15 & 2.42\\
         8  & 1.93/2.03 & 1.84/1.99 & \textbf{1.84}/2.17 & 1.87/2.32 & 2.60/2.49 & 3.00/2.41 & 2.23\\
        10  & \textbf{1.56}/1.94 & 1.63/1.99 & 1.60/2.08 & 1.74/2.12 & 2.79/3.29 & 2.40/2.29 & 3.23\\
        12  & 1.31/1.53 & \textbf{1.26}/1.50 & 1.32/1.63 & 1.53/1.67 & 2.54/2.79 & 2.50/3.22 & 2.35  \\
        14  & 1.83/1.98 & \textbf{1.83}/2.08 & 1.86/2.02 & 2.01/2.05 & 3.34/3.69 & 3.18/2.97 & 2.57\\
        16  & 1.45/2.22 & \textbf{1.43}/2.18 & 1.54/2.07 & 1.68/2.16 & 3.31/3.98 & 2.33/2.56 & 2.25\\
        18  & \textbf{1.37}/2.16 & 1.42/2.04 & 1.42/2.10 & 1.61/2.20 & 2.61/3.12 & 2.36/2.67 & 2.60\\
        20  & \textbf{1.41}/2.13 & 1.44/2.15 & 1.46/2.12 & 1.46/1.98 & 3.27/3.10 & 2.52/2.73 & 2.84\\
        \bottomrule
    \end{tabular}
\end{table}

\subsection{Experiment 2 and 3: Amenable vs non-amenable} 

The results of Experiment 2 are summarized in Tables \ref{tab:precision_combined} and \ref{tab:mse_combined} for precision and MSE, respectively. For precisions, one can clearly observe that figures in general favour the case when true graphs are amenable. The MSE results, on the other hand, are a bit surprising where algorithms seem to perform better for non-amenable graphs. One possible explanation for this is that non-amenable graphs are in general `denser' than amenable graphs, therefore, the true causal effect/variance of adjustment sets distribute unevenly.

\begin{table}[htbp]
    \centering
    \caption{Combined Precision results (Amenable / Non-amenable), Experiment 2.}
    \label{tab:precision_combined}
    \sisetup{round-mode=places, round-precision=2}
    \begin{tabular}{r c c c c c}
        \toprule
        $n$ & $s=1$ & $s=0.8$ & $s=0.5$ & $s=0.3$ & $s=0$ \\
        \midrule
        6  & \textbf{0.90}/0.75 & 0.85/0.71 & 0.80/0.64 & 0.71/0.55 & 0.58/0.41  \\
        8  & \textbf{0.77}/0.76 & 0.76/0.75 & 0.73/0.69 & 0.68/0.59 & 0.54/0.41  \\
        10 & \textbf{0.89}/0.72 & 0.88/0.71 & 0.85/0.65 & 0.82/0.61 & 0.60/0.44 \\
        12 & \textbf{0.77}/0.70 & 0.76/0.71 & 0.73/0.69 & 0.71/0.66 & 0.44/0.45 \\
        14 & \textbf{0.81}/0.78 & 0.80/0.76 & 0.79/0.73 & 0.76/0.70 & 0.50/0.44  \\
        16 & \textbf{0.79}/0.74 & 0.78/0.73 & 0.75/0.72 & 0.74/0.69 & 0.49/0.41 \\
        18 & 0.76/0.73 & \textbf{0.76}/0.71 & 0.74/0.67 & 0.71/0.66 & 0.39/0.41  \\
        20 & \textbf{0.70}/0.64&0.70/0.64 & 0.68/0.59 & 0.69/0.58 & 0.37/0.32 \\
        \bottomrule
    \end{tabular}
\end{table}

\begin{table}[htbp]
    \centering
    \caption{Combined MSE$\times 10^{-2}$ results (Amenable / Non-amenable), Experiment 2.}
    \label{tab:mse_combined}
    \begin{tabular}{r c c c c c c}
        \toprule
        $n$ & $s=1$ & $s=0.8$ & $s=0.5$ & $s=0.3$ & $s=0$ & $\opadj(\G,x,y)$ \\
        \midrule
        6  & \textbf{1.42}/1.31 & 1.55/1.44 & 1.59/1.74 & 1.80/2.10 & 2.30/2.64 & 0.64/0.66 \\
        8  & 1.25/\textbf{1.17} & 1.30/1.23 & 1.23/1.46 & 1.35/1.83 & 1.91/2.43 & 0.71/0.60\\
        10 & \textbf{1.33}/1.98 & 1.44/1.97 & 1.37/2.26 & 1.64/2.19 & 3.05/2.61 & 0.66/0.65\\
        12 & 1.63/\textbf{1.30} & 1.71/1.36 & 1.89/1.48 & 2.12/1.54 & 3.20/2.24 & 0.73/0.70\\
        14 & 1.38/\textbf{1.22} & 1.43/1.28 & 1.47/1.41 & 1.65/1.49 & 2.33/3.13 & 0.69/0.67\\
        16 & 1.35/1.35 & \textbf{1.34}/1.36 & 1.41/1.45 & 1.44/1.65 & 2.32/2.74 & 0.79/0.62\\
        18 & 1.82/\textbf{1.31} & 1.81/1.35 & 1.79/1.53 & 2.04/1.51 & 3.22/2.94 & 0.72/0.73\\
        20 & 1.48/\textbf{1.30}& 1.49/1.34 & 1.49/1.45 & 1.57/1.46 & 2.73/3.21  & 0.80/0.74\\
        \bottomrule
    \end{tabular}
\end{table}

\begin{figure}[htbp]
    \centering
    \includegraphics[width=0.6\linewidth]{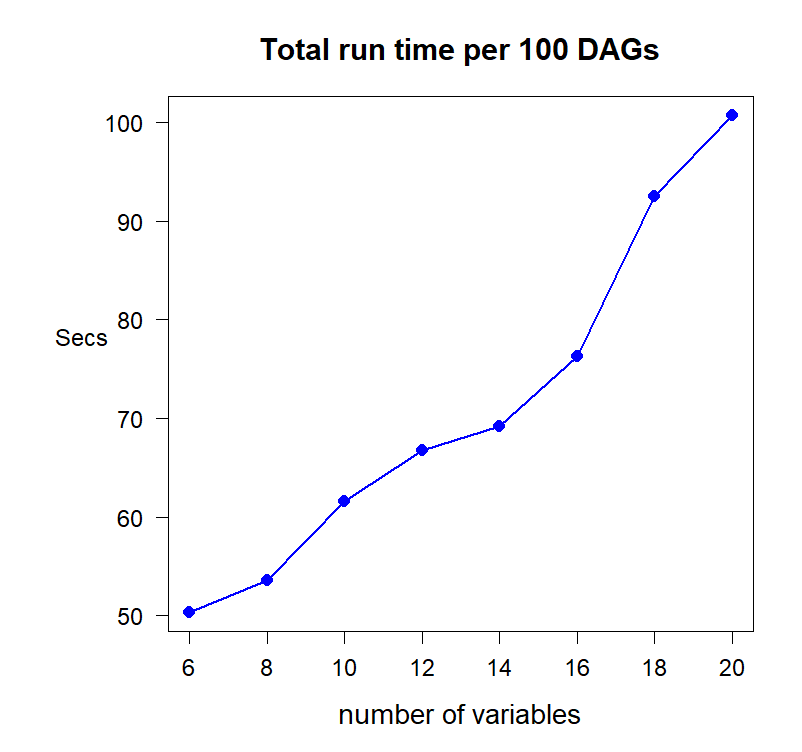}
    \caption{Total run time per 100 DAGs as a function of the number of nodes, displaying a linear tendency.}
    \label{fig: time plot}
\end{figure}

Figures \ref{fig: conv_sample_size plot} and \ref{fig: conv_sampling plot} present the results of Experiment 3. One can clearly observe that as the number of samples approaches infinity, the precision approaches 1 for amenable graphs while this behavior is not present for non-amenable graphs. A similar pattern appears when we increase the sampling time. This again justifies our result in Proposition \ref{prop: convergence result}.

\begin{figure}[htbp]
    \centering
    \includegraphics[width=0.8\linewidth]{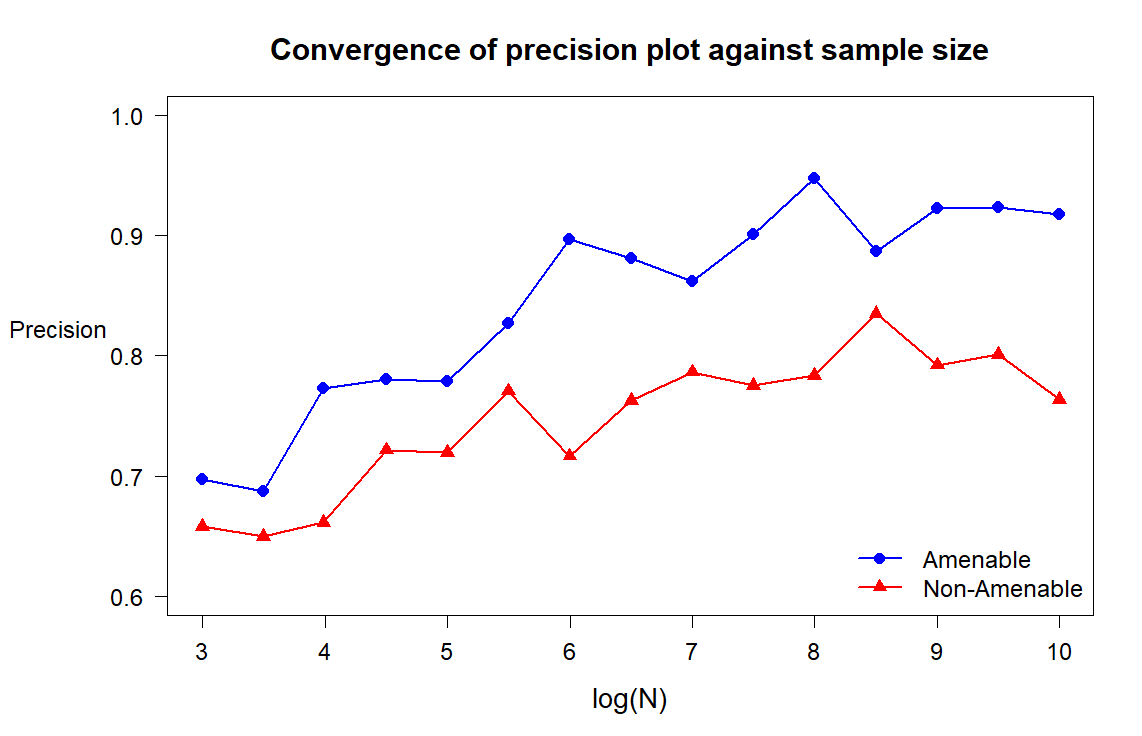}
    \caption{Precision plot against sample size. Precision is higher for amenable DAGs. Experiment 3.}
    \label{fig: conv_sample_size plot}
\end{figure}

\begin{figure}[htbp]
    \centering
    \includegraphics[width=0.8\linewidth]{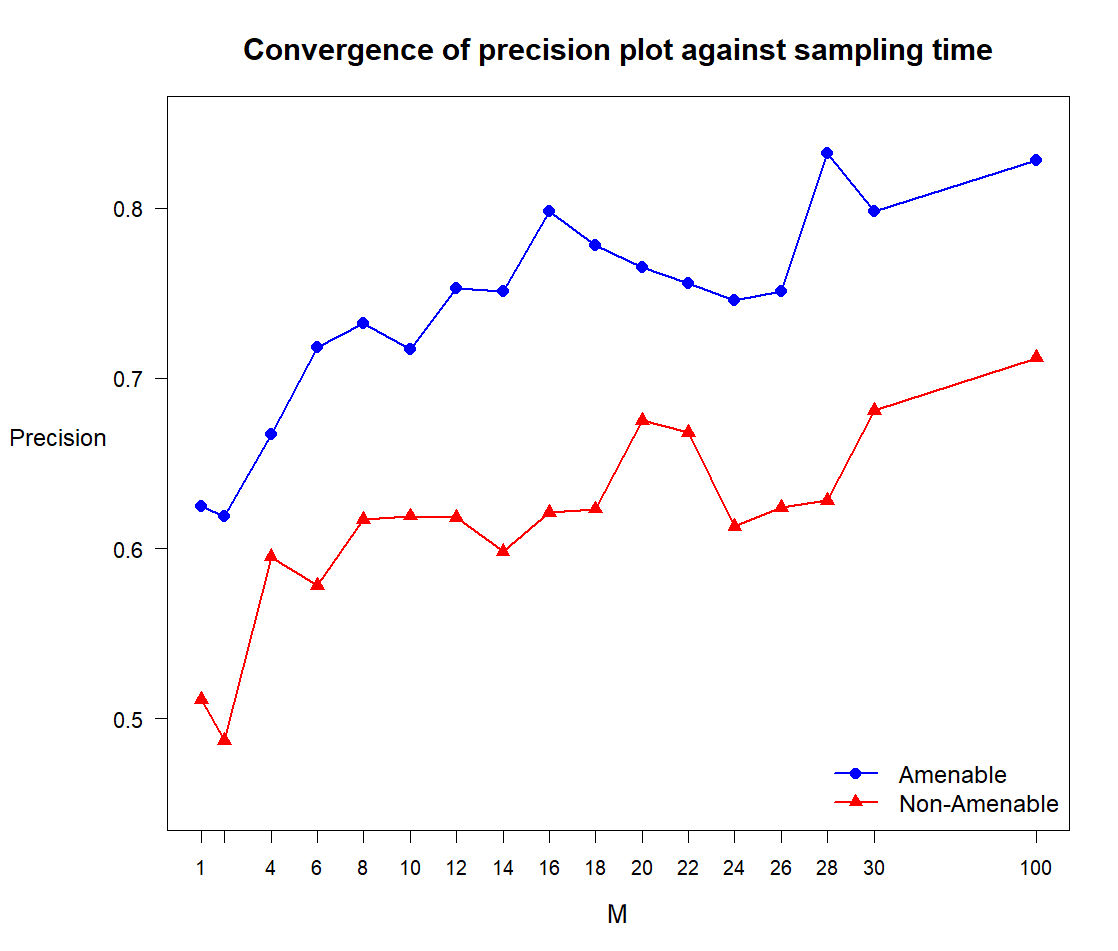}
    \caption{Precision plot against sampling time. Precision is higher for amenable DAGs. Experiment 3.}
    \label{fig: conv_sampling plot}
\end{figure}

\subsection{Real data example: Sachs}\label{Sec: real world example}
In this section, we apply our algorithm to a real-world data set from \citet{sachs2005causal}. The data set is a widely used dataset in causal inference. It was introduced by Sachs et al. (2005) to study complex relationships in cellular signaling pathways. The data set consists of 7466 points with 11 variables representing different proteins and molecules that are part of the signaling network. In Figure \ref{fig: sachs plot}, we provide the most commonly accepted ground-truth DAG \citet{kleinegesse2022domain} for the Sachs data set. 

There are different theories about the ground truth DAG. The main difference is whether there is any edge from $\text{Plcg}, \text{PIP3}, \text{PIP2}$ to the rest of the nodes. A study on recovering the DAG from the data by \citet{bnlearn_sachs} shows that even if there are, they are weak edges. Therefore, we choose the DAG shown in Figure \ref{fig: sachs plot}. As input to our algorithm, we will use only the skeleton. 

There is no obvious choice for the treatment $x$ and the outcome $y$. To explore a pair of nodes where non-causal confounding paths appear, we choose the treatment $x$ to be Mek and the outcome $y$ to be Erk. There is a directed path/edge from Mek to Erk, as well as confounding paths such as $\text{Mek} \leftarrow \text{PKA} \rightarrow \text{Erk}$. In this case, the optimal adjustment set is $\{\text{PKA}\}$. 

To avoid the algorithm getting stuck and to allow it to explore more freely, we reduce the data size to 100 with multiple runs and count how many times for each set appear at the top of the list. A similar procedure is performed for the GES algorithm. Our algorithm finds that $\{\text{PKA}\}$ appears the most often, more than twice as often as the set $\{\text{PKA},\text{Akt}\}$. The GES algorithm, on the other hand, identifies $\{\text{PKA},\text{Akt}\}$ as the most possible valid adjustment set.

The GES with known skeleton often produces DAGs similar to the DAG in Figure \ref{fig: wrong sachs figure} in the appendix, which is not the ground-truth DAG and consider Akt as one parent of Erk. The resulting empirical optimal adjustment set is not a valid adjustment set, which is $\opadj(\G^{'},x,y)=\{\text{Akt},\text{PKA}\}$. This shows that sampling techniques are preferred. 
\begin{figure}
    \centering
    \includegraphics[width=0.6\linewidth]{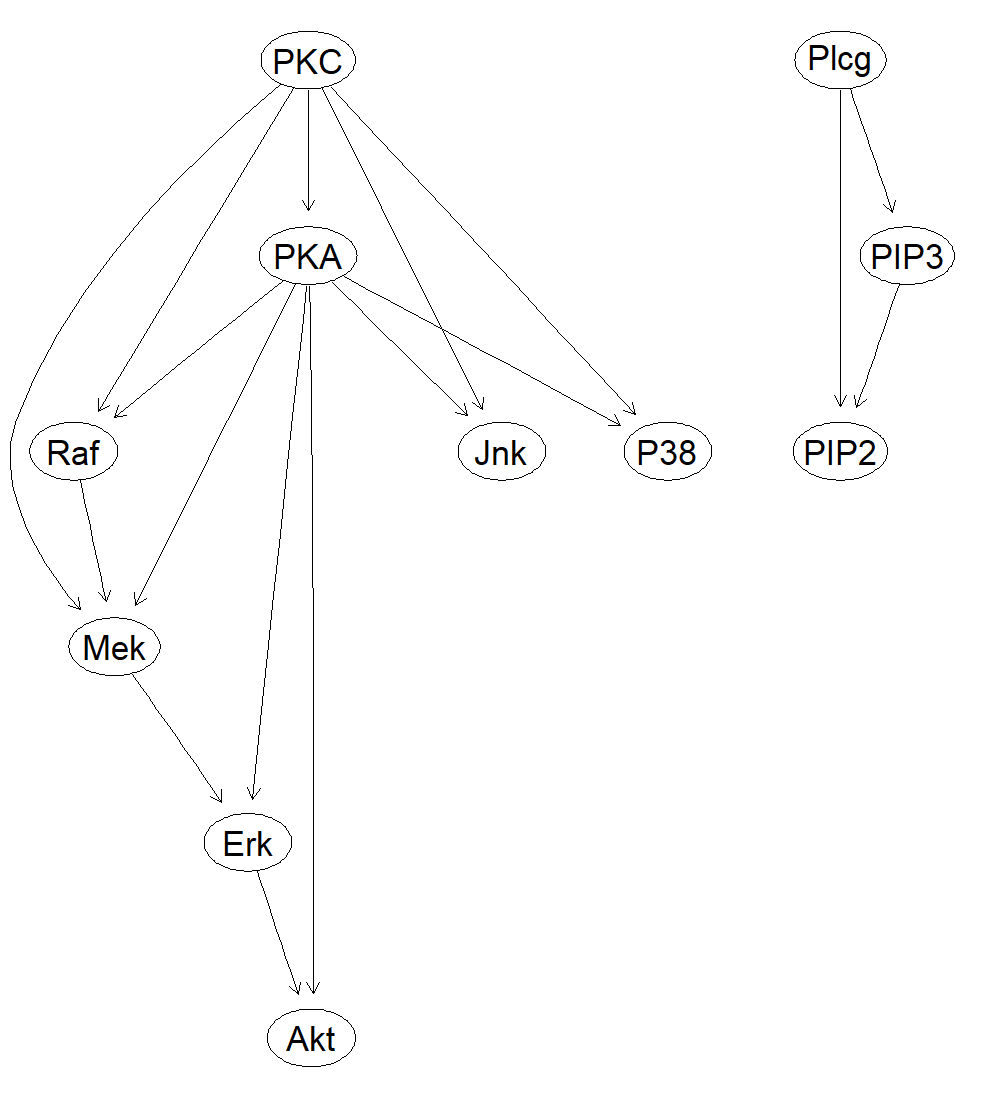}
    \caption{Ground truth DAG for Sachs data set.}
    \label{fig: sachs plot}
\end{figure}
\section{Discussion}\label{sec:discussion}
In this work, we present a method to select adjustment sets while not knowing the full underlying causal DAG. The method relies on two important features: sampling graphs and testing adjustment sets. While we mainly focus on the testing sets process, our contributions are twofold: (i) we offer a framework to identify valid adjustment sets with higher precision than traditional methods that researchers would now use in practice; (ii) we establish some theoretical foundation and practical techniques to accelerate the testing sets process, which would be exponential if done in the naive way. Proposition \ref{prop:efficient adjustment set} not only ensures that even if a limited number of sets are tested, the output sets still tend to be valid and efficient, but also can be generalized to more relaxed graphical assumptions rather than a known and fixed skeleton.

For the sampling procedure, we work with a linear Gaussian graphical model and fixed skeleton because this is one of the most basic models (yet common), and we are able to build theory and algorithms, and present our idea more clearly. However, one can easily choose to work with different parametric graphical assumptions. The results of Proposition \ref{prop:invariant adjustent set} and Lemma \ref{lemma: good node} may not be applied directly, but Proposition \ref{prop:efficient adjustment set} can be used to compute candidates of adjustment sets from each sampled graph. An important note for the sampling procedure is that we do not want the algorithms to converge to a single DAG; instead, the sampling algorithms should explore around the space of the `best' DAG, which allows those `stable' adjustment sets to appear and be identified.

Some other practical methods have also been studied to reduce the empirical error from a single output of DAG learning algorithms. For example, one may bootstrap the data and compute multiple DAGs, then decide on the orientation of each edge by voting from these DAGs \citep{bnlearn_sachs}. However, this method is not suitable for estimating valid adjustment sets, as merely assembling the edges may result in a cyclic graph. If one computes valid adjustment sets from these DAGs and let them vote the most frequent set, this method is still less efficient than using sampled DAGs, as Section \ref{sec:adj change when graph change} shows, we can quickly check whether some sets remain valid/invalid based on previously sampled DAGs.

\section{Acknowledgement}
This work is funded by the European Union. Views and opinions expressed are, however, those of the author(s) only and do not necessarily reflect those of the European Union or the European Research Council Executive Agency. Neither the European Union nor the granting authority can be held responsible for them. This work is supported by ERC grant BayCause, nr. 101074802.
\bibliographystyle{abbrvnat} 
\bibliography{references}       

\newpage

\begin{appendix}
\section{Appendix}

An initial step towards avoiding examining every set is that we can only consider adjustment sets that precede $Y$ in the topological ordering. This can be justified by the following proposition.

Let $L_Y$ be the set of nodes that succeeds $Y$ in a given topological order. We assume that $X$ always precedes $Y$ in any topological ordering.
\begin{proposition}\label{prop:precedingtopological}
Suppose $A$ is a valid adjustment set for $X$ and $Y$, then $A \setminus L_Y$ is also a valid adjustment set.    
\end{proposition}
\begin{proof}
Suppose $A' := A \setminus L_Y$ is not valid. Certainly, criterion \eqref{criteria 1} is met, so there must exist a non-causal path from $X$ to $Y$ that is blocked by $A$ but not by $A'$. 

Consider any such non-causal path $\pi$. If there is no collider, then it must be a confounding path, i.e.\ $X \leftarrow  \dots  \leftarrow ? \rightarrow  \dots \rightarrow Y$, but then every node on the path precedes $X$ and $Y$, and hence whether it is blocked is not affected by removing $L_Y$.

Suppose there is at least one collider on the path $\pi$. Because this path is not blocked by $A'$, it means that every collider on the path $\pi$ is contained in $\an(A') \subseteq \an(A)$, thus, if it is blocked by $A$ but not $A'$,  there must be some nodes on the path that are non-colliders and are in $A \setminus A' = L_Y$. Let $C$ denote all the colliders on the path. For any non-collider on the path, it must be in $\an(C \cup \{X,Y\}) \subseteq \an(A' \cup\{X,Y\})$. But, $L_Y \cap \an(A' \cup \{X,Y\}) = \emptyset$, therefore, we arrive at a contradiction.
\end{proof}

 \begin{figure}[htbp]
\centering
 \begin{tikzpicture}
  [rv/.style={circle, draw, thick, minimum size=6mm, inner sep=0.8mm}, node distance=20mm, >=stealth]
  \pgfsetarrows{latex-latex};
\begin{scope}
  \node[rv]  (1)            {$\text{Raf}$};
  \node[rv, below of=1] (2) {$\text{Mek}$};
   \node[rv, above right of=1] (8) {$\text{PKA}$};
  \node[rv, above of=8] (9) {$\text{PKC}$};
  \node[rv, below right of=8] (11) {$\text{Jnk}$};
  \node[rv, right of=11] (10) {$\text{P38}$};
  \node[rv, right of=10] (4) {$\text{PIP2}$};
  \node[rv, above of=4] (3) {$\text{Plcg}$};
  \node[rv, above right of=4] (5) {$\text{PIP3}$};
  \node[rv, below right of=2] (6) {$\text{Erk}$};
  \node[rv, below right of=6] (7) {$\text{Akt}$};
  \draw[->, very thick, blue] (2) -- (1);
\draw[->, very thick, blue] (8) -- (1);
\draw[->, very thick, blue] (9) -- (1);
\draw[->, very thick, blue] (9) -- (2);
\draw[o-o, very thick, black] (4) -- (3);
\draw[o-o, very thick, black] (5) -- (3);
\draw[o-o, very thick, black] (5) -- (4);
\draw[->, very thick, blue] (2) -- (6);
\draw[->, very thick, blue] (7) -- (6);
\draw[->, very thick, blue] (8) -- (6);
\draw[->, very thick, blue] (8) -- (7);
\draw[->, very thick, blue] (2) -- (8);
\draw[->, very thick, blue] (9) -- (8);
\draw[->, very thick, blue] (8) -- (10);
\draw[->, very thick, blue] (8) -- (11);
\draw[->, very thick, blue] (9) -- (11);
\draw[->, very thick, blue] (9) -- (10);

  \end{scope}
\end{tikzpicture}
\caption{wrong CPDAG returned by GES }
\label{fig: wrong sachs figure}
\end{figure}
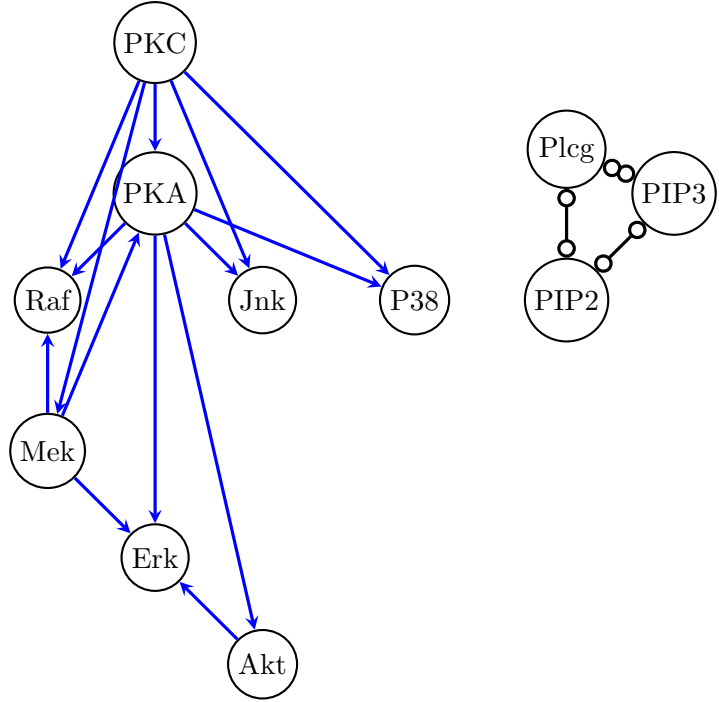   

\end{appendix}
\end{document}